\documentclass{article}
\usepackage{times}
\usepackage{amsmath, amssymb, amsthm, amscd}
\theoremstyle{plain}
\newtheorem{theorem}{Theorem}[section]

\newtheorem{definition}[theorem]{Definition}

\newtheorem{proposition}[theorem]{Proposition}

\begin{document}

\title{ITTMs with Feedback}
\author{Robert S. Lubarsky\\
Florida Atlantic University\\ Dept. of Mathematical Sciences
\\ 777 Glades Rd., Boca Raton, FL 33431, USA\\
Robert.Lubarsky@alum.mit.edu \\}

\maketitle

\begin {abstract}
Infinite time Turing machines are extended in several ways to
allow for iterated oracle calls. The expressive power of these
machines is discussed and in some cases determined.
\end {abstract}

\section {Introduction}
Infinite time Turing machines, or ITTMs, introduced in \cite {HL},
are regular Turing machines that are allowed to run for
transfinitely many steps. The only changes to the standard
definition of a Turing machine that need making are what to do at
limit stages: the head goes to the front of the tape(s), the state
entered is a dedicated state for limits, and the value of each
cell is the lim sup of the previous values.

That introductory paper also discussed various kinds of oracles
computations and corresponding jump operators. One such jump
operator encodes the information ``does the ITTM with index $e$ on
input $r$ converge?" If $e$ is allowed to call an oracle $A$, this
is called the {\bf strong jump} $A^{\blacktriangledown}$ of $A$:
$\{(e,x) \mid \{e\}^A(x)\downarrow\}$. The jump can of course be
used as an oracle itself, and the process iterated: you can, for
instance, ask whether $\{e\}(r)$ converges, where $\{e\}$ can
itself ask oracle questions of simple (non-oracle) ITTMs.

We would like to investigate ultimate iterations of this jump, for
several reasons. Iterations of a procedure can lead to new
phenomena. A well-known example of that in a context similar to
the current one is transfinite iterations of the regular Turing
jump. If you iterate the Turing jump along any well-order that
appears along the way, you get the least admissible set containing
your starting set, admissible computability theory being a quantum
leap beyond ordinary computability theory \cite{B}. Arguably the
next example right after this one would be iterations of inductive
definitions. Admissible set theory is exactly what you need to
develop a theory of positive inductive definitions, the least
fixed point of such being $\Sigma_1$ definable over any admissible
set containing the definition in question (e.g. its parameters)
\cite{B}. If the language of least fixed points of positive
inductive definitions is closed in a straightforward manner, you
end up with the $\mu$-calculus. Determining the sets definable in
the $\mu$-calculus is however anything but a straightforward
extension of admissibility, needing a generalization of the notion
of reflection, gap reflection \cite{L,M,L-}. Something similar
happens with ITTMs, as some of the extensions are quite different
from the base case, as we will see.

A potential application of this work is in proof theory. The
strongest fragment of second-order arithmetic for which an ordinal
analysis has been done to date is $\Pi^1_2$ Comprehension
\cite{R}. Regular (i.e. non-iterated) ITTMs are already more
powerful than that. Perhaps having descriptions of stronger
subsystems of analysis other than the straightforward hierarchy of
$\Pi^1_n$ Comprehension principles will help the proof theorists
make progress.

The goal of this line of inquiry is to examine what kind of
iterations of ITTMs make sense, and to quantify how powerful those
iterations are by characterizing the reals, or what amounts to the
same thing ordinals or sets, that can be so written. This
situation is different from that for regular Turing machines,
because an ITTM computation can halt after infinitely many steps,
and so ITTMs have the power to write reals. Hamkins and Lewis
insightfully classified the reals that come up in this context as
{\bf writable} if they appear as the output of a halting
computation, {\bf eventually writable} if they are eventually the
unchanging content of the output tape for a divergent computation,
and {\bf accidentally writable} if they appear anywhere on any
tape during any ITTM computation, even if they are overwritten
later. The same concepts apply to ordinals, where an ordinal is
writable (resp. eventually, accidentally) if some real coding that
order-type is writable (resp. eventually, accidentally). This
distinction among these kinds of reals and ordinals turned out to
be crucial to their characterization, as announced in \cite {W1}
and detailed in \cite {W2}, with improved proofs and other results
in \cite {W3}. Let $\lambda, \zeta$, and $\Sigma$ respectively be
the suprema of the writable, eventually writable, and accidentally
writable ordinals.

\begin {theorem}
(Welch) $\zeta$ is the least ordinal $\alpha$ such that
$L_{\alpha}$ has a $\Sigma_2$ elementary extension, $L_\lambda$ is
the smallest $\Sigma_1$ substructure of $L_\zeta$, and $L_\Sigma$
is the unique $\Sigma_2$ extension of $L_\zeta$.
\end {theorem}

The relativization of this theorem to a real parameter holds
straightforwardly.

In the next section, we give some notions of the syntax and
semantics of these iterations fundamental to what follows. The
three after that each gives a different kind of extension of
ITTMs, and about as much as is currently know about them. Some are
characterized pretty fully, others only to the point where it's
clear that there's something very different going on. The final
section offers a generalization of the semantics.

\section {Feedback ITTMs and the Tree of Subcomputations}

A {\bf feedback ITTM} ({\bf FITTM}) is an ITTM with two additional
tapes, and an additional state, which is the oracle query ``does
the feedback ITTM with program the content of the first additional
tape on input the content of the second converge?" Clearly, the
additional tapes are merely an expository convenience, as they
could be coded as dedicated parts of the original tape.

The semantics of feedback ITTMs is defined via the tree of
subcomputations. The idea is that the tree keeps track of oracle
calls by having each one be a child of the calling computation.
This tree is in general not itself ITTM computable. Rather, it is
defined within ZF, even if a fragment of ZF would suffice,
inductively on the ordinals. At every ordinal stage, each extant
node is labeled with some computation, and control is with one
node.

At stage 0, control is with the root, which we think of as at the
top of the downward growing subcomputation tree. The root is
labeled with the index (and input, if any) of the main
computation.

At a successor stage, if the node currently in control is in any
state other than the oracle call, action is as with a regular
Turing machine. If taking that action places that machine in a
halting state, then, if there is a parent, the parent gets the
answer ``convergent" to its oracle call, and control passes to the
parent. If there is no parent, then the current node is the root,
and the computation halts. If the additional step does not place
the machine in a halting state, then control stays with the
current node. If the current node makes an oracle call, a new
child is formed, after (to the right of) all of its siblings,
labeled with the calling index and parameter; a new machine is
established at that node, with program the given index and with
the parameter written on the input tape; and control passes to
that node.

At a limit stage, there are three possibilities. One is that on
some final segment of the stages there were no oracle calls, and
so control was always at one node. Then the rules for limit stages
of regular ITTMs apply, and the snapshot of the computation at the
node in question is determined (where the snapshot includes all of
the current information about the computation -- the state, the
tape contents, and so on). If that snapshot repeats an earlier
one, then that computation is divergent. (Here we are using the
standard convention, first articulated in \cite{HL}, that a
snapshot qualifies as repeating only if it guarantees an infinite
loop. In point of fact, a snapshot might be identical to an
earlier one, which guarantees that it will recur $\omega$-many
times, but it is possible that at the limit of those snapshots, we
escape the loop. So by convention, a repeating snapshot is taken
to be one that guarantees that you're in a loop.) At that point,
if there is a parent, then the parent gets the answer ``divergent"
to its oracle call, and control is passed to the parent. If there
is no parent, then the node in question is the root, and the
entire computation is divergent.

A second possibility is that cofinally many oracle calls were
made, and there is a node $\rho$ such that cofinally many of those
calls were $\rho$'s children. Note that such a node must be
unique. Then $\rho$ was active cofinally often, and again the
rules for regular ITTMs at limit stages apply. If $\rho$ is seen
at that stage to be repeating, then control passes to $\rho$'s
parent, if any, which also gets the answer that $\rho$ is
divergent; if $\rho$ is the root, then the main computation is
divergent. If $\rho$ is not seen to be repeating at this stage,
then $\rho$ retains control and the computation continues.

The final possibility is that, among the cofinally many oracle
calls made, there is an infinite descending sequence, which is the
right-most branch of the tree. This is bad. It is troublesome, at
best, to define what to do at the next step. Various ways to avoid
this last situation are the subject of the next sections.

\section {Pre-Qualified Iterations}
The problem cited above is that the subcomputation tree has an
infinite descending sequence. The most obvious way around that is
to ensure that that does not happen, that the tree is
well-founded. That can be enforced by attaching an ordinal to each
node of the tree and requiring that children of a node have
smaller ordinals.

That is in essence what is done with the strong jump
$\emptyset^\blacktriangledown$ of \cite{HL}.
$\emptyset^\blacktriangledown$ is $\{ (e,x) \mid e(x)\downarrow
\}$, which is the same thing as labeling the root of the
subcomputation tree with 1, so none of its children, the oracle
calls, can themselves make oracle calls. In unpublished work, Phil
Welch has show that $\zeta^{\emptyset^\blacktriangledown}$ is the
smallest $\Sigma_2$-extendible limit of $\Sigma_2$-extendibles,
and that $\lambda^{\emptyset^\blacktriangledown}$ and
$\Sigma^{\emptyset^\blacktriangledown}$ are such that
$L_{\lambda^{\emptyset^\blacktriangledown}}$ is the least
$\Sigma_1$ substructure of
$L_{\zeta^{\emptyset^\blacktriangledown}}$, which is itself a
$\Sigma_2$ substructure of
$L_{\Sigma^{\emptyset^\blacktriangledown}}$.

We would like to generalize this to ordinals as large as possible,
certainly to ordinals greater than 1. An {\bf ordinal oracle ITTM}
is an FITTM with not two but three additional tapes. On the third
tape is written a real coding an ordinal $\alpha$. The oracle
calls allowed are about other ordinal oracle ITTMs, and on the
third tape must be written some ordinal $\beta < \alpha$. Since
one of the other tapes is for parameter passing, it is unimportant
just how the ordinals are written on the latest tape. With this
restriction, the third outcome above can never happen, and all
computations are well-defined (as either convergent or divergent).

An {\bf iterated ITTM}, or {\bf IITTM}, is an FITTM that may make
an oracle call about any ordinal oracle ITTM writing on the third
tape any ordinal at all. So an IITTM is like an ordinal oracle
ITTM only the length of the ordinal iteration is not fixed in
advance. Rather, it is limited only by what the machine figures
out to write down.

\begin {definition}
$\lambda^{it}, \zeta^{it},$ and $\Sigma^{it}$ are the respective
suprema of the ordinals writable, eventually writable, and
accidentally writable by IITTMs.
\end {definition}

\begin {definition}
An ordinal $\alpha$ is \begin {itemize} \item 0-extendible if it
is $\Sigma_2$ extendible, \item $\beta+1$-extendible if it is a
$\Sigma_2$ extendible limit of $\beta$-extendibles, and
\item $\gamma$-extendible ($\gamma$ a limit) if it is $\Sigma_2$
extendible and a limit of $\beta$-extendibles for each $\beta <
\gamma$.
\end {itemize}

\end {definition}

As pointed out by the referee, the limit clause actually works
perfectly well for all three clauses.

The definition above relativizes to any parameter $x$. The
corresponding notation is for $\alpha$ to be
$\beta[x]$-extendible. Notice that, in the limit case, when
$\gamma < \alpha$, $\alpha$ is also the limit of ordinals which
are themselves limits of $\beta$-extendibles for each $\beta <
\gamma$.

\begin {theorem}
For ordinal oracle ITTMs with ordinal $\alpha$ coded by the input
real $x_\alpha$ and parameter $y$, the supremum $\zeta$ of the
eventually writable ordinals is the least
$\alpha[x_\alpha,y]$-extendible. Moreover, the supremum $\Sigma$
of the accidentally writable ordinals is such that
L$_\Sigma[x_\alpha,y]$ is the (unique) $\Sigma_2$ extension of
L$_\zeta[x_\alpha,y]$, and the supremum $\lambda$ of the writable
ordinals is such that L$_\lambda[x_\alpha,y]$ is the smallest
$\Sigma_1$ substructure of L$_\zeta[x_\alpha,y]$. Finally, the
writable (resp. eventually, accidentally) reals are those in the
corresponding segment of L$[x_\alpha,y]$.
\end {theorem}
\begin {proof}
By induction on $\alpha$.

$\alpha = 0$: This is the relativized version of Welch's theorem
cited above.

$\alpha = \beta + 1$: Let $\gamma$ be any ordinal less than
$\zeta$. Run some machine which eventually writes $\gamma$.
Dovetail that computation with the following. Simulate running all
ordinal oracle ITTMs with input $\beta$ and as parameters the
output of the first machine, which is eventually $\gamma$, and
$y$. This is essentially running a universal machine: clear
infinitely many cells on the scratch tape, split them up into
countably many infinite sequences, and on the $i^{th}$ sequence
run a copy of the $i^{th}$ machine. For each of those simulations,
keep asking whether the current output will ever change. (That is,
ask whether the computation that continues that simulation until
the output tape changes, at which point it halts, is convergent.)
This is a legitimate question for the oracle, as $\beta < \alpha$.
Whenever you get the answer ``no," indicate as much on a dedicated
part of the output tape. Eventually you will get all and only the
indices of the eventually stable computations. So the least
$\beta[x_\alpha,y]$-extendible ordinal is less than $\zeta$, and
so $\zeta$ is the limit of such.

Because of this closure under $\beta$-extendibility,
L$_\zeta[x_\alpha,y]$ can run correctly the computation of the
ordinal oracle ITTMs with input $\beta$. So the rest of the proof
-- that the computations of eventually writable reals stabilize by
$\zeta$, and that the eventually writable reals form a $\Sigma_2$
substructure of the accidentally writables and a $\Sigma_1$
extension of the writables -- follows by the same arguments
pioneered in \cite{W2} and improved upon in \cite {W3}. In order
to keep this paper self-contained, and to verify that the new
context here really makes no difference, we present these
arguments here.

Suppose, toward a contradiction, that L$_\zeta[x_\alpha,y]$
satisfies some $\Pi_2$ sentence $\phi$, but L$_\Sigma[x_\alpha,y]$
does not. By the nature of $\Pi_2$ sentences, the set of ordinals
$\xi \leq \Sigma$ such that L$_\xi[x_\alpha,y] \models \phi$ is
closed, and so contains its maximum. By hypothesis, that maximum
is strictly less than $\Sigma$. Take some machine that
accidentally writes each of the ordinals less than $\Sigma$. A
universal machine will do, for instance, so we will call this
machine $u$. We also need a machine, say $p$, which eventually
writes the $\phi$'s parameter. It is safe to assume that there is
only one parameter, as finitely many can be combined into one set
by pairing. If no parameter is necessary, then $\emptyset$ as a
dummy parameter can be used. Our final machine, call it $e$, runs
$p$ and $u$ simultaneously. It takes the output of $u$ and uses it
to generate the various L$_\xi[x_\alpha,y]$s. When it finds such a
set modeling $\phi$, with parameter the current output of $p$, it
compares $\xi$ to the current content of the output tape. If the
current content is an ordinal greater than or equal to $\xi$,
nothing is written and the computation continues. Else $\xi$ is
written on the output. Eventually the output of $p$ settles down.
Once that happens, when the largest such $\xi$ ever appears, it
will be so written, after which point it will never be
overwritten, making $\xi$ eventually writable. This is a
contradiction.

Regarding $\lambda$, suppose L$_\zeta[x_\alpha,y]$ satisfies some
$\Sigma_1$ formula $\psi$ with parameters from
L$_\lambda[x_\alpha,y]$. Consider the computation which first
computes the parameters using a halting computation, then runs a
machine which eventually writes a witness to $\psi$ and halts when
it finds one. This is a halting computation for such a witness.

By the foregoing, $\zeta$ is $\alpha[x_\alpha,y]$-extendible. That
it is the least such is ultimately because the assertion that any
particular cell in a computation stabilizes is $\Sigma_2$. In
detail, let $\zeta_\alpha$ be the least
$\alpha[x_\alpha,y]$-extendible ordinal and $\Sigma_\alpha$ its
$\Sigma_2$ extension. Since stabilization is a $\Sigma_2$
assertion, any computation has the same eventually stable cells at
$\zeta_\alpha$ as at $\Sigma_\alpha$. Moreover, if $\delta$ is a
stage beyond which a certain cell is stable in $\zeta_\alpha$, the
assertion that that cell beyond $\delta$ is stable is $\Pi_1$, so
that same $\delta$ is also a stabilization point in
$\Sigma_\alpha$. So the snapshot of a computation at
$\zeta_\alpha$ is that same at $\Sigma_\alpha$, and all looping
has occurred by then.

$\alpha$ a limit: Since ordinal oracle ITTMs with input $\alpha$
subsume those with input $\beta < \alpha$, $\zeta$ is
$\beta[x_\alpha,y]$-extendible for each $\beta < \alpha$, and
hence, considering successor $\beta$s, a limit of
$\beta[x_\alpha,y]$-extendibles. The rest follows as above.
\end {proof}

\begin {theorem} $\zeta^{it}$ is the least $\kappa$ which is
$\kappa$-extendible, $\lambda^{it}$ its smallest $\Sigma_1$
substructure, and $\Sigma^{it}$ its (unique) $\Sigma_2$ extension.
\end {theorem}
\begin {proof}
For every $\alpha < \zeta^{it}$, the ordinal oracle ITTMs with
input $\alpha$ are also IITTMs. Hence the least
$\alpha$-extendible is $\leq \zeta^{it}$, and $\zeta^{it}$ is a
limit of $\alpha$-extendibles. The rest, again, follows as above.
\end {proof}

\section {Freezing Computations}
Another way to deal with the possible ill-foundedness of the
subcomputation tree is not to worry about it. That is, while no
steps are taken to rule out such computations, there will be some
with perfectly well-founded subcomputation trees, even if only by
accident. We remain positive, and focus our attention on those,
where we have a well-defined semantics, including whether a
computation converges or diverges. So we can define the reals
writable, eventually writable, and accidentally writable by
FITTMs.

\begin {proposition} Every feedback eventually writable real is
feedback writable.
\end {proposition}
\begin {proof}
Let $e$ be a computation which writes a feedback eventually
writable real. Consider an alternative computation which runs $e$
on a dedicated part of the tapes. Every time $e$'s output tape
changes, the main computation asks the oracle: ``Consider the
computation which begins at the current snapshot of $e$, and
continues $e$'s computation until the output tape changes once
more, and then halts. Does that converge or diverge?" Since $e$'s
tree of subcomputations is well-founded, so is that of the oracle
call, and the oracle call will return a definite answer. If that
answer is ``converge," then the construction continues; if
``diverge", then the construction halts. By hypothesis, this
computation eventually halts, at which point $e$'s output is
written on the output tape.
\end {proof}
Even worse:
\begin {proposition} Every feedback accidentally writable real is
feedback writable.
\end {proposition}
\begin {proof}
Suppose $e$ is a divergent computation. As in \cite{HL}, $e$ then
has to loop, and does so already at some countable stage. The
sledgehammer way to see that is that there are only set-many
possible snapshots, so if a computation never halts then it has to
repeat itself. As to why that would happen at some countable
stage, that follows from Levy absoluteness. More concretely, the
argument in \cite{HL} for regular ITTMs applies unchanged in the
current setting. There are only countably many cells. So only
countably many stop changing beneath $\aleph_1$. Moreover, there
is some countable bound $\alpha$ by which those have all stopped
changing. List the remaining cells in an $\omega$-sequence $c_0,
c_1, ...$. Let $\alpha_0$ be the least stage beyond $\alpha$ at
which $c_0$ changes. Inductively, let $\alpha_n$ be the least
stage beyond $\alpha_{n-1}$ by which all of $c_0, c_1, ..., c_n$
have changed since stage $\alpha_{n-1}$. The configuration at
stage $\alpha_\omega = \lim_n \alpha_n$ repeats unboundedly
beneath $\aleph_1$, and so is a looping stage.

Let $\alpha$ be such that $e$ has already started to loop by
$\alpha$ many steps. Suppose we could write (a real coding)
$\alpha$ via a halting computation. Then any real written at any
time during $e$'s computation would be writable, via the program
``write $\alpha$, then compute $e$ for the number of steps given
by the integer $n$ in the coding of $\alpha$, then output
whatever's on $e$'s tapes then" (with the desired choice of $n$,
of course). So it suffices to write the looping time of a
computation.

First we determine the first looping snapshot of the machine. At
every stage of the computation in a simulation of $e$, the oracle
is asked: ``Consider the computation that begins with the current
snapshot of $e$, saves it on a dedicated part of the tape, and
continues with a simulation of $e$ on a different part of the
tape, halting whenever the original snapshot is reached again;
does this computation halt?" If the answer is ``no," the
simulation continues. Eventually the answer will be ``yes." That
is the first looping snapshot. (Actually, as pointed out in
\cite{HL}, that's not quite right. A snapshot can repeat itself,
which would then force it to repeat $\omega$-many times, but the
limit could be unequal to that repeating snapshot, and so this
loop could be escaped. The constructions here could be modified
easily enough to avoid this problem.)

The next thing to do would be to write the ordinal number of steps
it took to get to that looping snapshot, and the ordinal number of
steps it would take to make one loop, and then to add them. Since
those ordinals are constructed the same way, we will describe only
how to do the second.

During the construction, we will assign integers to ordinals in
such a way that the $<$-relation will be immediate. The
construction will take $\omega$-many stages, during each of which
we will use up countably (or finitely) many integers, so
beforehand assign to each $n \in \omega$ countably many integers
disjointly to be available at stage $n$. Furthermore, each integer
has its own infinite part of the tape for its scratchwork.

Let $C_i$ ($i \in \omega$) be the (simulated) $i^{th}$ cell of the
tape on which we're running (the simulation of) $e$. We will need
to know which cells change value cofinally in the stage of
interest (the return of the looping stage) and which don't. So
simulate the run of $e$ from the looping stage until its
reappearance. Every time $C_i$ changes value, toggle the $i^{th}$
cell on another dedicated tape from 0 to 1 to 0. At the end of the
computation, the $i^{th}$ cell on the dedicated tape will be 0 iff
$C_i$ changed value boundedly often; so it will be 1 iff $C_i$
changed value cofinally often.

Stage 0 starts in the looping snapshot, and is itself split into
$\omega$-many steps. Those steps interleave consideration of the
cells that changed boundedly often and those that change
cofinally. At step $2i$ continue the computation until the
$i^{th}$ cell with bounded change stops changing. That can be
determined by asking the oracle whether the cell in question
changes before the looping snapshot reappears. While this is not a
converges-or-diverges question on the face of it, since the
computation converges in any case (either when the cell changes or
when the looping snapshot is reached, whichever happens first),
one of those outcomes can be changed to a trivial loop, so that
the question is a standard oracle call. If the answer is ``yes,"
then continue the computation until the answer becomes ``no,"
which is guaranteed to happen. At that point, use an available
integer to mark that ordinal stage, which integer is then larger
in the ordinal ordering than all other integers used so far. Also
write the current snapshot in that integer's scratchwork part of
the tape. Then proceed to the next step, $2i+1$.

At step $2i+1$, we will consider not just the $i^{th}$ cofinally
changing cell, but also the $j^{th}$ such for all $j \leq i$, for
purposes of dovetailing. Sequentially for each $j$ from 0 to $i$,
go to the next stage at which the $j^{th}$ cofinally changing cell
changes value again. After doing so for $i$, use an available
integer to mark that ordinal stage, which integer is then larger
in the ordinal ordering than all other integers used so far. Also
write the current snapshot in that integer's scratchwork part of
the tape. Then proceed to step $2(i+1)$.

Because stage 0 consists of $\omega$-many steps, each of which
picks out only one integer in an increasing sequence, it picks out
a strictly increasing $\omega$-sequence of ordinals. The limit of
that ordinal sequence is the ordinal in the computation at which
its looping snapshot reappears. That's because by then we're
beyond the ordinal at which any cell with boundedly many changes
will change again, thanks to the even steps, and those cells with
cofinal changing change cofinally in that ordinal, thanks to the
dovetailing in the odd steps.

To summarize, we have produced an $\omega$-sequence cofinal in the
ordinal at which the looping snapshot reappears. Inductively,
suppose at stage $i>0$ we have an integer assignment, with $<$, to
a subset of $e$'s ordinal stage, as well as a picture of the
snapshot of the computation each at such stage of the computation.
Then for each integer which is a successor in this partial
assignment, replicate the construction above with the starting
snapshot being the snapshot of $e$ at the predecessor and the
ending snapshot being the snapshot of $e$ at the integer under
consideration. By the well-foundedness of the ordinals, this
process ends after $\omega$-many stages.

\end {proof}

It is easy to see that the feedback writable reals are those
contained in the initial segment of L given by the feedback
writable ordinals, which are also the FITTM clockable ordinals. We
call the set of these ordinals $\Lambda$.

This result removes the basis of the analysis used in weaker forms
of ITTM computation. It comes about because the divergence of a
computation in this paradigm can be determined convergently by a
computation of the same type. Why doesn't this run afoul of some
kind of diagonalization result? The answer is that there's no
universal machine! That is, the computations and oracle calls used
in the proofs above were sometimes convergent and sometimes
divergent, but conveniently they were in any case all
well-defined: the tree of subcomputations was well-founded. If it
is not, we have no semantical notion of how the computation should
continue or what the outcome should be. This notion is captured in
the following.

\begin {definition} A computation is {\bf freezing} if its tree of
subcomputations is ill-founded.
\end {definition}

\begin {proposition} There is no FITTM computation which decides
on an input $e$ whether the $e^{th}$ FITTM is freezing.
\end {proposition}
\begin {proof} If there were, you could diagonalize against the
non-freezing computations, for a contradiction.
\end {proof}

We expect that as with most models of computation, the key to
understanding what's computable will be an analysis of the
uncomputable. While the freezing computations do not have an
output or even a divergent computation, they are perfectly
well-defined up until the point when an oracle call is made about
a freezing subcomputation. For that matter, on the tree of
subcomputations, that freezing subcomputation generates a good
tree underneath it, until it calls its own freezing
subcomputation. More generally, even for a freezing computation,
its subcomputation tree, albeit ill-founded, is well-defined.
Hence the following definition makes sense.

\begin {definition} A real is {\bf freezingly writable} if it
appears anywhere on a tape during a freezing computation or any of
its subcomputations.
\end {definition}

We expect that the role that the eventually and accidentally
writable reals played in the understanding of the writable reals
for basic ITTMs will be played here by the freezingly writable
reals. In any case, it should be of interest to understand better
the freezing computations. Centrally, what does the subcomputation
tree of a freezing computation look like? Since the computation
cannot continue once an infinite path through the tree develops,
that infinite path is unique, and is the right-most path. So each
of the $\omega$-many levels on the tree has width some successor
ordinal. For each freezing computation $e$, let $\lambda^e_n$ be
the width of level $n$ of $e$'s subcomputation tree. For a fixed
$e$, there are three possibilities for the $\lambda^e_n$s:

a) $\lambda^e_n$ is bounded beneath $\Lambda$.

b) $\lambda^e_n$ is cofinal in $\Lambda$.

c) Some $\lambda^e_n$ is greater than $\Lambda$.

Option a) is simply unavoidable: it is a simple task to write a
machine which immediately makes an oracle call about itself,
producing a subcomputation tree of order-type $\omega^*$ ($\omega$
backwards).

Options b) and c), as it turns out, are incompatible with each
other. To see this, first note that if c) holds for some
computation, then $n$ can be chosen to be 1 (level 0 consisting of
the root alone). After all, if this is not the case for some given
$e$, let $e_1$ be some computation that halts at a stage larger
than $\max_{m<n}\lambda^e_m$. Use $e_1$ to write $e_1$'s run-time
(using methods like those in the main proposition above). Use that
ordinal to run $e$ substituting for the oracle calls an explicit
computation until the right-most node on level $n-1$ (of $e$'s
original subcomputation tree) becomes active. That is the node
which has more than $\Lambda$-many children, and which is now the
root node of the tree of this modified computation.

Now assume we have indices $e_b$ and $e_c$ of types b and c
respectively (and $\lambda^{e_c}_1 > \Lambda$). Simulate $e_c$.
Whenever an oracle call is made, write the new length of the top
level in the subcomputation tree (using techniques as above). Use
that ordinal to simulate the computation of $e_b$ substituting
explicit computation for oracle calls and building explicitly the
subcomputation tree. Whenever the run of $e_b$ demands an ordinal
greater than that provided by $e_c$ yet, break off the former
computation and return to the latter. By hypothesis, at a certain
point you will be able see that $e_b$'s subcomputation tree is
ill-founded. Then write $\sup_{n \in \omega} \lambda^{e_b}_n$, and
halt. This would then be a halting computation of $\Lambda$,
contradiction.

Unfortunately, we do not know which of b) or c) is excluded. For
that matter, there could be no examples of either! Possibly all
freezing computations are of type a), where those bounds over all
freezing $e$s are cofinal in $\Lambda$.

\section {Parallel Oracle Calls}
With sequential computation, as defined above, once an ill-founded
oracle call is made, the entire computation is freezing. Parallel
computation provides an alternative. In its essence, this is the
same as with finite computation. In that setting, what should be
the semantics of ``A or B"? That both converge and one is true, or
that one is true regardless of whether the other even converges?
Similarly here, a machine could make a parametrized oracle call.
This is perhaps most easily modeled by having another tape as part
of the oracle call. The called computation asks for the
convergence of a computation with index given on the first tape
and inputs the second and third tapes. When making a call, the
third tape is blank, but in generating the answer, the oracle
substitutes all possible finite strings (equivalently: all
integers) on the blank tape. If any return a convergent
computation, the oracle answers ``yes." If none of them freeze and
all return a divergent computation, the oracle answers ``no." If
at least one of the parallel calls freezes and all those that do
not diverge, then the oracle gives no answer and the current
computation freezes.

Notice that the roles of convergence and divergence could be
interchanged here, as convergent and divergent computations can be
interchanged with each other: given $e$, ask the oracle whether
$e$ converges; if yes, diverge, if no, halt. Of course, if $e$
freezes, so does this.

Arguments similar to those above show that the parallel writable,
parallel eventually writable, and parallel accidentally writable
reals are all the same.

Although it seems likely, we do not have a proof that the parallel
writable reals include strictly more than the feedback writables
do.

\section {Extending Convergence and Divergence Consistently}
For both (sequential) feedback and parallel computation above, the
semantics was given conservatively. That is, the
convergence/divergence answers to oracle calls were forced on us.
Evidence for such was an explicit computation in which some tree
was well-founded, as so is absolute. Once well-foundedness is
brought into the picture, induction cannot be too far behind. In
fact, the process can be described via an inductive definition.

Let $\downarrow$ and $\uparrow$ be a disjoint pair of sets of
computation calls, where a computation call is a pair consisting
of (an index for) a program and a parameter. Given ($\downarrow,
\uparrow$), computations can be defined as convergent or divergent
relative to that pair. For the sake of concreteness we will
restrict attention to feedback computation; analogous
considerations apply to parallel computation. When making oracle
calls, the given pair ($\downarrow, \uparrow$) is used as the
oracle. This is deterministic, as $\downarrow$ and $\uparrow$ are
disjoint. It is also monotonic: any computation that asks only
oracle calls already in $\downarrow$ or $\uparrow$ will be
unaffected by increasing either or both of those; all other
computations are freezing, and so can only thaw by increasing
those. As a monotonic operator, it has a least fixed point. This
is the semantics given for feedback computation, that is the sense
in which the semantics was conservative. This description of the
matter does allow for considering other fixed points as possible
semantics for these computational languages.

\end {document}